\documentclass[11pt]{article}
\usepackage[margin=1.0in]{geometry}

\usepackage{verbatim}

\usepackage{amsmath}
\usepackage{amsthm}
\usepackage{amssymb}
\usepackage{amsfonts}
\usepackage{multicol}
\usepackage{hyperref}
\usepackage{subfigure}
\usepackage[numbers,sort&compress]{natbib}
\usepackage{booktabs}
\usepackage{titlesec}

\usepackage{enumitem}
\setlist[enumerate]{label={(\arabic*)}}

\usepackage{color}
\newcommand{\red}[1]{\textcolor{black}{#1}}
\newcommand{\blue}[1]{\textcolor{black}{#1}}
\newcommand{\Hstar}{{\cal H}^{*}}
\newcommand{\calH}{{\cal H}}

\usepackage{tikz}
\usetikzlibrary{calc}

\usepackage{tkz-graph}

\newtheorem{theorem}{Theorem}[section]
\newtheorem{lemma}[theorem]{Lemma}

\newtheorem{claim}[theorem]{Claim}
\newtheorem{obs}[theorem]{Observation}

\theoremstyle{definition}

\newtheorem{problem}{Open Problem}

\newcommand{\squishlist}{
 \begin{list}{$\bullet$}
  { \setlength{\itemsep}{0pt}
     \setlength{\parsep}{3pt}
     \setlength{\topsep}{3pt}
     \setlength{\partopsep}{0pt}
     \setlength{\leftmargin}{2.5em}
     \setlength{\labelwidth}{1em}
     \setlength{\labelsep}{0.5em} } }

\newcommand{\squishlisttwo}{
 \begin{list}{$\triangleright$}
  { \setlength{\itemsep}{0pt}
     \setlength{\parsep}{0pt}
    \setlength{\topsep}{0pt}
    \setlength{\partopsep}{0pt}
    \setlength{\leftmargin}{2em}
    \setlength{\labelwidth}{1.5em}
    \setlength{\labelsep}{0.5em} } }

\newcommand{\squishend}{
  \end{list}  }
\newcommand{\forbid}{P_3+\ell P_1}
\begin{document}

\title{Vertex-critical $(P_3+\ell P_1)$-free and vertex-critical (gem, co-gem)-free graphs}
\author{Tala Abuadas\\
\small Department of Physics and Computer Science\\
\small Wilfrid Laurier University\\
\small Waterloo, ON Canada\\
\and
Ben Cameron\\
\small Department of Computing Science\\
\small The King's University\\
\small Edmonton, AB Canada\\
\small ben.cameron@kingsu.ca\\
\and
Ch\'{i}nh T. Ho\`{a}ng\\
\small Department of Physics and Computer Science\\
\small Wilfrid Laurier University\\
\small Waterloo, ON Canada\\
\small choang@wlu.ca\\
\and
Joe Sawada\\
\small School of Computer Science\\
\small University of Guelph\\
\small Guelph, ON Canada\\
\small jsawada@uoguelph.ca\\
}

\date{\today}

\maketitle

\begin{abstract}

A graph $G$ is $k$-vertex-critical if $\chi(G)=k$ but $\chi(G-v)<k$ for all $v\in V(G)$ where $\chi(G)$ denotes the chromatic number of $G$. We show that there are only finitely many $k$-critical $(P_3+\ell P_1)$-free graphs for all $k$ and all $\ell$. Together with previous results, the only graphs $H$ for which it is unknown if there are an infinite number of $k$-vertex-critical $H$-free graphs is $H=(P_4+\ell P_1)$ for all $\ell\ge 1$. We consider a restriction on the smallest open case, and show that there are only finitely many  $k$-vertex-critical (gem, co-gem)-free graphs for all $k$, where gem$=\overline{P_4+P_1}$. To do this, we show the stronger result that every vertex-critical (gem, co-gem)-free graph is either complete or a clique expansion of $C_5$. This characterization allows us to give the complete list of all $k$-vertex-critical (gem, co-gem)-free graphs for all $k\le 16$.


\end{abstract}
\section{Introduction}

For a given $k\ge 3$, determining the $k$-colorability graphs is a classic NP-complete problem~\cite{Karp1972}. When the graphs are from a special family, however, polynomial-time algorithms have been developed to determine the $k$-colorability for some (and sometimes all) values of $k$. Families of particular interest are hereditary families defined by one or more forbidden induced subgraphs. A foundational result to this end is the polynomial-time algorithm to determine the $k$-colorability of perfect graphs~\cite{Grotschel1984}. On the negative side, determining the $k$-colorability of $H$-free graphs for any $k\ge 3$ remains NP-complete if $H$ contains an induced cycle~\cite{Maffray1996, KaminskiLozin2007} or 
%
%
claw~\cite{Holyer1981,LevenGail1983}. Thus, assuming P$\neq$NP, if $k$-colorability of $H$-free can be determined in polynomial time for some $k\ge 3$, then $H$ must be the disjoint union of paths.  Seinsche \cite{Seinche} proved that 
%
$P_4$-free graphs are perfect and therefore their $k$-colorability is polynomial-time solvable. A recursive polynomial-time algorithm exploiting dominating structures to determine the $k$-colorability of $P_5$-free graphs for all $k$ was developed in~\cite{Hoang2010}. The strength of this algorithm was further highlighted when it was shown that determining the $k$-colorability of $P_t$-free graphs is NP-complete if $t\ge 7$ and $k\ge 4$ or $t=6$ and $k\ge 5$~\cite{Huang2016}. In the same work, it was also conjectured that the $4$-colorability of $P_6$-free graphs could be determined in polynomial time. This conjecture was proved in a series of two preprints \cite{P6free1,P6free2} (see also a concise conference paper of the results~\cite{P6freeconf}). This completed the complexity dichotomy for determining $4$-colorability of $H$-free graphs when $H$ is connected. For smaller $k$, the $3$-colorability of $P_7$-free graphs can be determined in polynomial time~\cite{Bonomo2018} and the complexity remains an open question for $P_t$-free when $t\ge 8$.

When we consider $H$-free graphs for disconnected $H$, things are still very interesting. The polynomial-time algorithm to determine $k$-colorability of $P_5$-free graphs for any $k$ was later generalized for $(P_5+rP_1)$-free graphs for any $r\ge 0$ in \cite{Couturier2015}. The $3$-colorability of $(P_6+rP_3)$-free graphs can determined in polynomial time for any $r\ge 0$, while for $k\ge 5$, $k$-colorability of $(P_5+P_2)$-free graphs is NP-complete \cite{ChudnovskyHuangSpirklZhong2021}. Very recently it was shown that the $5$-colorability of $rP_3$-free graphs can be determined in polynomial time, while the $k$-colorability of $(r+1)P_4$-free graphs remains NP-complete for any $r\ge 1$ and $k\ge 5$ \cite{HajebiLiSpirkl2021}. It follows that the graph $H$ of smallest order for which the complexity of determining the $k$-colorability of $H$-free graphs is not known for all $k$ is $H=P_3+P_2$, where it is polynomial-time solvable for $k\le 5$ and has unknown complexity for $k\ge 6$.

Beyond the complexity of $k$-colorability algorithms, there is also interest in developing ones that are certifying. An algorithm is \textit{certifying} if together with each output it includes a simple and easily verifiable witness that the output is correct. In the case of $k$-colorability algorithms, a $k$-coloring serves as a certificate for a positive answer, and an induced $(k+1)$-vertex-critical subgraph can be returned to certify a negative answer. It is well-known that if a hereditary family of graphs $\mathcal{F}$ contains only finitely many $(k+1)$-vertex-critical graphs, then a polynomial-time algorithm to determine the $k$-colorability of any graph in $\mathcal{F}$ can be readily implemented by searching for each of the $(k+1)$-vertex-critical graphs as induced subgraphs and returning one as a certificate if found (see \cite{P5banner2019} for more details). In the past decade there has been a push for the development of more certifying algorithms in general, including the strong stance taken by McConnell et al.~\cite{McConnell2011} that ``for complex algorithmic tasks, only certifying algorithms are satisfactory.'' While many of the $k$-colorability algorithms cited above return $k$-colorings if they exists, in most of these cases, how to efficiently return vertex-critical subgraphs remains either unknown or impossible \cite{ChudnovskyStacho2018}. 
However, a linear-time certifying algorithm for determining the $3$-colorability of $P_5$-free graphs was developed by showing that there are exactly 12 $4$-vertex-critical $P_5$-free graphs \cite{Bruce2009,MaffrayGregory2012}. Unfortunately, the following theorem asserts that this positive trend does not continue for larger $k$. 

\begin{theorem}[\cite{Hoang2015}]\label{thm:infif2K2}
If $H$ is not $2K_2$-free and $k \geq 5$, then there is an infinite number of $k$-vertex-critical $H$-free graphs.
\end{theorem}

This has led to significant interest in classifying $k$-vertex-critical $(P_t,H)$-free graphs for $t\ge 5$. 
There are exactly $13$ $5$-vertex-critical $(P_5,C_5)$-free graphs~\cite{Hoang2015} and only finitely many $6$-vertex-critical $(P_5,\text{banner})$-free graphs where \textit{banner }is the graph obtained from a cycle of order four by attaching a leaf to one of its vertices~\cite{P5banner2019}. There are only finitely many $4$-vertex-critical $(P_8,C_4,C_5)$-free graphs~\cite{ChudnovskyStacho2018} and $5$-vertex-critical $(P_6,\text{banner})$-free graphs~\cite{Huang2019}. 
More generally, there are only finitely many $k$-vertex-critical $(P_5,H)$-free graphs for all $k \geq 1$ if $H=\overline{P_5}$~\cite{Dhaliwal2017} or $H=K_{s,s}$ (the complete bipartite graph of order $2s$) for any $s\ge 1$~\cite{KaminskiPstrucha2019} where the latter extends an analogous result for $(P_6,C_4)$-free graphs~\cite{Hell2017} and the former includes a structural characterization.
In recent work, it was shown for $k \geq 5$ that there are only finitely many $k$-vertex-critical $(P_5,H)$-free graphs for $H$ of order four if and only if $H$ is neither $2P_2$ nor $K_3+P_1$~\cite{Cameron2020}.  Very recently, it was shown that there are only finitely many $k$-vertex-critical $(P_5,H)$-free graphs for all $k$ when $H=\text{gem}$ (where gem$=\overline{P_4+P_1}$) or $H=\overline{P_3+P_2}$~\cite{CaiGoedgebeurHuang2021}. Returning to $H$-free vertex-critical graphs, substantial progress was recently made with the following dichotomy theorem.

\begin{theorem}[\cite{Chud4critical2020}]\label{thm:finite4critical}
Let $H$ be a graph. There are only finitely many $4$-vertex-critical $H$-free  graphs if and only if $H$ is an induced subgraph of $P_6$, $2P_3$, or $P_4+\ell P_1$ for some $\ell \in\mathbb{N}$.
\end{theorem}

It follows that the only open cases of which graphs $H$ there are only finitely many $k$-vertex-critical $H$-free graphs for all $k$ are $H=P_n+\ell P_1$ for $1\le n\le 4$ and $\ell\ge 1$.  For $n=1$ and any $\ell\ge 1$, the finiteness follows trivially by Ramsey's Theorem. For $n=2$ and any $\ell\ge 1$, the second, third, and fourth authors showed finiteness \cite{CameronHoangSawada2020}. For $n=3$ and $\ell=1$, finiteness was shown in \cite{Cameron2020}. In this paper, we further resolve some of the remaining open cases by showing the following:

\begin{itemize}
\item There are only finitely many $k$-vertex-critical $(\forbid)$-free graphs for all $k\ge 1$ and $\ell \ge 0$.
\item There are only finitely many $k$-vertex-critical (gem, co-gem)-free graphs for all $k\ge 1$, and moreover, every such graph must be complete or a clique-expansion of $C_5$.
\end{itemize}
%
%
\subsection{Outline}
We prove our result on $(\forbid)$-free graphs in Section~\ref{sec:P3+ellP1} and our results on (gem, co-gem)-free graphs in Section~\ref{sec:gemcogem}. The precise structural characterization of $k$-vertex-critical (gem, co-gem)-free graphs allows us to exhaustively generate (with aid of computer search) all such graphs for $k\le 16$ which is outlined in Section~\ref{subsec:exhaustivegemcogemgeneration}. We conclude with some open problems in Section~\ref{sec:conclusion}. Before getting to any of that, we first list many of the definitions and notations that will be used throughout the paper.

\subsection{Definitions and notation}\label{sec:definitions}
Let $G$ be a graph and $a,b$ be two vertices of $G$. We write $a \sim b$ to mean $a$ is
adjacent to $b$, and $a \nsim b$ otherwise.   A $k$-coloring $c(G)$ of $G$ is a mapping $c : V(G) \rightarrow \{1,2,\ldots, k\}$ such that $c(x) \not= c(y)$ if $x \sim y$; $c(x)$ is called the {\em color} of $x$. For a set $A$ of vertices of $G$ and a coloring $c(G)$, $c(A)$ denotes the set of colors that appear in $A$. The chromatic number of $G$, denoted by $\chi(G)$ is the smallest $k$ such that $G$ admits a $k$-coloring. The graph $G$ is \textit{$k$-vertex-critical} if $\chi(G)=k$ and $\chi(G-v)<k$ for all $v\in V(G)$. For a set $A \subseteq V(G)$, $G[A]$ denotes the subgraph of $G$ induced by $A$.

$P_n$ denotes the induced path of order $n$. $K_n$ denotes the clique on $k$ vertices.  For graphs $G$ and $H$, $G+H$ denotes the disjoint union of $G$ and $H$. For a positive integer $\ell$, let $\ell G$ denote the disjoint union of $\ell$ copies of $G$. The complement of a graph $G$ is denoted $\overline{G}$. For $S\subseteq V(G)$, we let $G-S$  denote the graph obtained from $G$ by deleting all vertices in $S$ along with their incident edges. For $v\in V(G)$, we let $G-v$ denote $G-\{v\}$; and we let $N(v)$ denote the set of vertices of $G-v$ that are adjacent to $v$. 

The \textit{gem} is the graph $\overline{P_4+P_1}$ and the \textit{co-gem} is $P_4+P_1$. A graph $G$ is \textit{perfect} if $\chi(H)=\omega(H)$ for all induced subgraphs $H$ of $G$. For subsets $A$ and $B$ of $V(G)$, we say $A$ is \textit{complete to $B$} if $ab\in E(G)$ for all $a\in A$ and $b\in B$ and we say $A$ is \textit{anti-complete to $B$} if $ab\not\in E(G)$ for all $a\in A$ and $b\in B$. For $M\subseteq V(G)$, we say $M$ is a \textit{module} if for all $v\in V(G)\setminus M$, $v$ is either complete or anti-complete to $M$. A module $M$ is \textit{non-trivial} if $M\neq V(G)$ and $|M|\neq 1$. Given a graph $G$ of order $n$ with vertices $v_1,v_2,\ldots,v_n$ and any disjoint non-empty graphs $H_1,H_2,\ldots,H_n$, the \textit{expansion} $G(H_1,H_2,\ldots,H_n)$ is the graph obtained from $G$ by replacing, for each $i$, $v_i$ with $H_i$ and joining $x\in H_i$ and $y\in H_j$ with an edge if and only if $v_i$ is adjacent to $v_j$ in $G$. Note that the expansion depends on an order of the vertices. An expansion is called a \textit{$P_4$-free expansion} if each $H_i$ is $P_4$-free and it is called a \textit{clique expansion} if each $H_i$ is a clique. For a given coloring $c$ of a graph $G$ and a subset $S\subseteq V(G)$, let $c(S)$ denote the set of all colors used on vertices in $S$ and $S^j$ denote the set of all vertices in $S$ colored $j$ provided $j$ is a color used in $c$. A \textit{stable set} is a subset $S$ of $V(G)$ such that $S$ induces no edges. We let $\alpha(G)$ denote the largest order of a stable set in $G$. A clique $C\subseteq V(G)$ in a graph $G$ is \textit{maximal} if $C\cup \{v\}$ is not a clique for all $v\in V(G)\setminus C$. A stable set $S$ of a graph $G$ is \textit{very good} is $S\cap K\neq\emptyset$ for every maximal clique $K$ of $G$. 

Let $R(r,s)$, for $r,s \geq 1$,  denote the Ramsey numbers, where $R(r,s)$ is the least positive integer such that every graph with at least $R(r,s)$ vertices contains either a clique on $r$ vertices or an independent set on $s$ vertices. We note that $R(r,s)$ always exists by Ramsey's Theorem~\cite{Ramsey}.

%

\section{($P_3 + \ell P_1$)-free graphs}\label{sec:P3+ellP1}
Throughout this section, assume $G$ is a $k$-vertex-critical ($P_3 + \ell P_1$)-free graph.

Let $S$ be a maximum stable set of $G$. Then, each vertex in $G-S$ has at least one neighbor in $S$. Let $A$ be the vertices in $G-S$ with exactly one neighor in $S$. Let $B = G - (S \cup A)$. We note that every vertex in $B$ has at least two neighbors in $S$. Let $S_A$ be the set of vertices $s \in S$ such that some vertex in $A$ is adjacent to $s$. Let $S_B = S - S_A$.

\begin{lemma}\label{lem:small-ss}
For $k\ge 3$ and $\ell\ge 0$,	$\alpha(G) <(k-1)^2 (\ell +3) $
\end{lemma}
\begin{proof}
By way of contradiction, assume $k\ge 3$ and $\ell\ge 0$ but $|S| \geq (k-1)^2 (\ell + 3)$. For a vertex $s \in S_A$, let $s_A$ be the neighbors of $s$ in $A$, i.e., $s_A=N(s)\cap A$. 

We will establish a number of claims before proving the Lemma.
\begin{claim}\label{cla:s_a}
For each $s \in S_A$, $s_A$ induces a clique in $G$. 
\end{claim}

\begin{proof}[Proof of Claim~\ref{cla:s_a}]
Let $s\in S_A$ and suppose $a_1,a_2\in s_A$ such that $a_1\nsim a_2$.  By definition, $a_1$ and $a_2$ are not adjacent to any vertex in $S-\{s\}$, so the set containing $a_1,a_2$, $s$ and any $\ell$ vertices in $S-\{s\}$ induces a $\forbid$, a contradiction.
\end{proof}

\begin{claim}\label{cla:no-edge} 
For any two vertices $x,y \in S_A$, $x_A$ is anti-complete to $y_A$.
\end{claim}
\begin{proof}[Proof of Claim~\ref{cla:no-edge}]
Suppose $x,y\in S_A$, $a_1\in x_A$ and $a_2\in y_A$ such that $a_1\sim a_2$. Then the set containing $a_1,a_2,x$  and any $\ell$ vertices in $S-\{x,y\}$ induces a $\forbid$, a contradiction.
\end{proof}

\begin{claim}\label{cla:Bnonempty}
$B \neq\emptyset$. 
\end{claim}
\begin{proof}[Proof of Claim~\ref{cla:Bnonempty}]
Suppose $B=\emptyset$. Therefore, $S_A=S$ and, by assumption, $|S_A|\ge (k-1)^2(l+3)$. Let $s\in S_A$. By Claim~\ref{cla:s_a}, $s_A$ induces a clique. Further, since $G$ is $k$-vertex-critical, $\deg(s)\ge k-1$, so $N[s]$ induces a clique of order at least $k$. Since $G$ is $k$-vertex-critical, it follows that $G=K_k$ and therefore $\alpha(G)=1$, a contradiction.
\end{proof}

\begin{claim}\label{cla:Batmostellnonneighbors}
Each vertex in $B$ has at most $\ell -1$ non-neighbors in $S$.
\end{claim}
\begin{proof}[Proof of Claim~\ref{cla:Batmostellnonneighbors}]
If there is a vertex in $B$ with at least $\ell$ non-neighbors in $S$, then the set containing this vertex together with any two of its neighbors in $S$ and any $\ell$ of its non-neighbors in $S$ induces a $\forbid$, a contradiction.
\end{proof}

\noindent  We say that a color class $i$ is \textit{big} if it has at least $(k-1) (\ell+3)$ vertices in $S$, otherwise we say that the color class is \textit{small}. 

\begin{claim}\label{cla:somecolorclassisbig}
Fix a $k$-coloring of $G$. If some color class has exactly one vertex, then some other color class is big.
\end{claim}
\begin{proof}[Proof of Claim~\ref{cla:somecolorclassisbig}]
If some color class has one vertex and each of the other $k-1$ color classes contain at most $(k-1) (\ell+3)-1$ vertices from $S$, then  
\begin{align*}
|S|&\leq (k-1)((k-1)(\ell+3)-1)+1\\
&=(k-1)^2(\ell +3)-k+2\\
&<(k-1)^2(\ell +3),
\end{align*}
a contradiction. Therefore, some color class has to be big. 
\end{proof} 

\begin{claim}\label{cla:small-color-class}
If a color appears in $B$, then its color class is small.
\end{claim}
\begin{proof}[Proof of Claim~\ref{cla:small-color-class}]
If a big color class contains $b\in B$, then $b$ has at least $(k-1) (\ell+3)>\ell-1$ non-neighbors in $S$, contradicting Claim~\ref{cla:Batmostellnonneighbors}.
\end{proof}

\noindent Now let $x$ be a vertex in $S$ with the least number of neighbors in $A$ among all vertices in $S$, that is, $|x_A|=\min_{s\in S}(|s_ A|)$. In particular, if $S_B \not= \emptyset$, then $x \in S_B$. Since $G$ is $k$-vertex-critical, we may fix a $(k-1)$-coloring of $G-x$ with colors from $\{1,2,\ldots, k-1\}$ such that all $k-1$ colors must appear in $N(x)$, for otherwise $G$ is $(k-1)$-colorable.  From Claim~\ref{cla:somecolorclassisbig}, some color class had to be big.

If $x\in S_B$, then all $k-1$ colors including one in a big color class appear in $B$, contradicting Claim~\ref{cla:small-color-class}. Therefore, $x \in S_A$ and $S_B=\emptyset$. Since $x_A$ is a clique, it has $|x_A|$ distinct colors. We may assume colors $1, \dots, t$ appear in $x_A$, and colors $t+1, \ldots, k -1$ appear in $N(x) \cap B$. If $t = k-1$, then $G$ contains a clique with $k$ vertices. Since $G$ is $k$-vertex-critical, $G$ is a clique on $k$ vertices. Thus we have $\alpha(G)=1$, a contradiction. Therefore, we may assume  $t<k-1$. The color classes of $t+1, \ldots, k -1$ must all be small by Claim~\ref{cla:small-color-class}.

Without loss of generality, assume color $1$ is big. Let $S_1$ be the set of vertices in $S_A=S$ with color $1$. For each vertex $s \in S_1$, the clique $s_A$ has at least as many vertices as $x_A$ by the choice of $x$. Thus, for each $s\in S_1$ there is a vertex $f(s) \in s_A$ with a color in the set ${\cal C} = \{ t+1, \ldots k -1\}$. Let $F$ be the set of all such vertices $f(s)$ for all $s\in S_1$. \red{By definition of $f(s)$, we have $F \subset A$.}
Since  $|F|=|S_1|\geq (k-1) (\ell+3)$ and $1\le |\mathcal{C}|\le k-2$, there are at least
$$\left\lceil\frac{(k-1)(\ell+3)}{k-2}\right\rceil>\ell
$$
vertices in $F$ that have the same color by the Pigeonhole Principle. Let $i$ be a color used on at least $\ell$ vertices in $F$  and let $I$ be a set of exactly $\ell$ vertices colored $i$ in $F$. By definition of $F$, $i\in {\cal C}$. Let $b_i$ be a vertex in $N(x) \cap B$ with color $i$. Recall that $b_i$ exists since colors $1,\ldots, t$ must appear in $x_A$ and colors $t+1,\ldots, k-1$ must appear in $N(x)\cap B$. \red{Since $I \subset A$, we have $b_i \not\in I$.}
\begin{claim}\label{cla:bitwoneighbors}
$b_i$ has at least two neighbors in $S-N(I)$.
\end{claim}
\begin{proof}[Proof of Claim~\ref{cla:bitwoneighbors}]
\red{By  Claim~\ref{cla:no-edge}, a vertex in $S$ cannot be adjacent to two vertices in $I$. It follows from the definition of $A$ that $|N(I)\cap S|=|I|=\ell$}. Therefore, $$|S-N(I)|\ge (k-1)^2(\ell+3)-\ell>\ell+1.$$ Thus, by Claim~\ref{cla:Batmostellnonneighbors}, $b_i$ has at least two neighbors in $S-N(I)$. 
\end{proof}
We now continue with the proof of the Lemma. 

\noindent Let $s_1,s_2\in S-N(I)$ such that $b\sim s_1$ and $b\sim s_2$.  Therefore, $\{b_i,s_1,s_2\}\cup I$ induces a $\forbid$, a contradiction. 
\end{proof}

\begin{theorem}\label{thm:finiteP3ellP1freecrit}
There are only finitely many $k$-vertex-critical $(\forbid)$-free graphs for all $k\ge 1$ and $\ell \ge 0$.
\end{theorem}
\begin{proof}
For $k=1,2$, the result is trivial. Fix $k>2$ and $\ell\ge 0$ and let $G$ be a $k$-vertex-critical $(\forbid)$-free graph. Since $G$ is $k$-colorable, $\omega(G)<k+1$. Therefore, from Lemma~\ref{lem:small-ss}, $|V(G)|<R(k,(k-1)^2 (\ell +3))$. Hence, by Ramsey's Theorem, every $k$-vertex-critical $(\forbid)$-free graph has order bounded by some constant depending only on $k$ and $\ell$. Therefore, there are only finitely many $k$-vertex-critical $(\forbid)$-free graphs.
\end{proof}

\section{(gem, co-gem)-free graphs}\label{sec:gemcogem}
From Theorem~\ref{thm:finiteP3ellP1freecrit}, the only graphs $H$ where it is unknown if there are only a finite number of $k$-vertex-critical $H$-free graphs for all $k$ is $H=P_4+\ell P_1$ for all $\ell\ge 1$. In this section we will consider a subclass of $(P_4+P_1)$-free graphs and show that this subclass contains only a finite number of $k$-vertex-critical graphs for all $k$. Our results in this section make extensive use of a structural characterization due to Karthick and Maffray \cite{KarthickMaffraygemcogem} whose statement requires a few definitions.

Let $G_1,G_2,\ldots,G_{10}$ be defined as shown in Figure~\ref{fig:Gis} and let $\mathcal{G}_i$ and $\mathcal{G}^{\ast}_i$ denote the set of all $P_4$-free and clique expansions of $G_i$, respectively. Throughout, for a graph $G\in \mathcal{G}_i$, we will use $A_i$ to denote the set of vertices used to replace $v_i$ in the graph $G_i$ according to the labellings in Figure~\ref{fig:Gis}.

\setcounter{subfigure}{0}
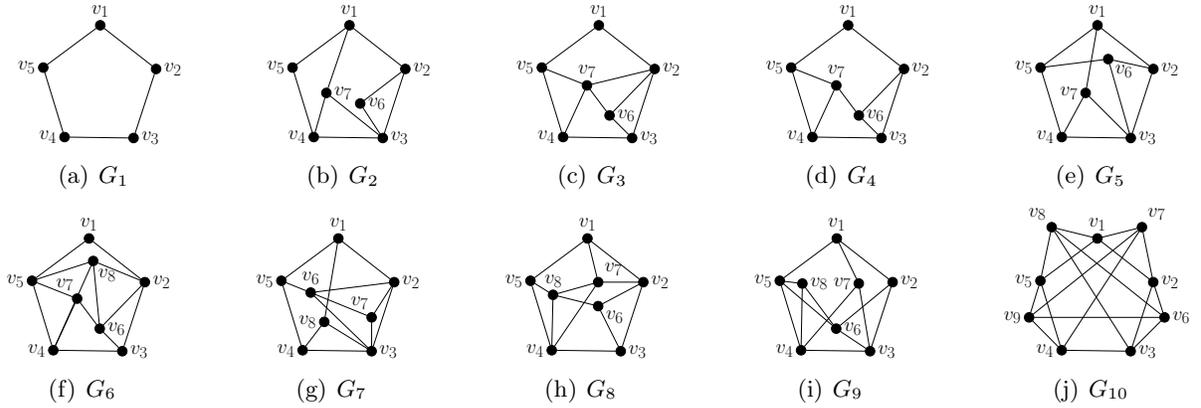
\begin{figure}[htb]
\def\c{0.3}
\def\r{1}
\centering
\subfigure[$G_1$]{
\scalebox{\c}{
\begin{tikzpicture}
\GraphInit[vstyle=Classic]
\Vertex[Lpos=180,L=\hbox{\Huge $v_5$},x=0.0cm,y=3.1235cm]{v0}
\Vertex[L=\hbox{\Huge $v_2$},x=5.0cm,y=3.0618cm]{v1}
\Vertex[Lpos=180,L=\hbox{\Huge $v_4$},x=0.942cm,y=0.0422cm]{v2}
\Vertex[Lpos=90,L=\hbox{\Huge $v_1$},x=2.5253cm,y=5.0cm]{v3}
\Vertex[L=\hbox{\Huge $v_3$},x=4.0006cm,y=0.0cm]{v4}

\Edge[](v0)(v2)
\Edge[](v0)(v3)
\Edge[](v1)(v3)
\Edge[](v1)(v4)
\Edge[](v2)(v4)
\end{tikzpicture}}}
\qquad
\subfigure[$G_2$]{
\scalebox{\c}{
\begin{tikzpicture}
\GraphInit[vstyle=Classic]
\Vertex[Lpos=180,L=\hbox{\Huge $v_5$},x=0.0cm,y=3.1235cm]{v4}
\Vertex[L=\hbox{\Huge $v_2$},x=5.0cm,y=3.0618cm]{v1}
\Vertex[Lpos=180,L=\hbox{\Huge $v_4$},x=0.942cm,y=0.0422cm]{v3}
\Vertex[Lpos=90,L=\hbox{\Huge $v_1$},x=2.5253cm,y=5.0cm]{v0}
\Vertex[L=\hbox{\Huge $v_3$},x=4.0006cm,y=0.0cm]{v2}
\Vertex[L=\hbox{\Huge $v_6$},x=3cm,y=1.53cm]{v5}
\Vertex[L=\hbox{\Huge $v_7$},x=1.5cm,y=2cm]{v6}
\Edge[](v0)(v1)
\Edge[](v0)(v4)
\Edge[](v0)(v6)
\Edge[](v1)(v2)
\Edge[](v1)(v5)
\Edge[](v2)(v3)
\Edge[](v2)(v5)
\Edge[](v2)(v6)
\Edge[](v3)(v4)
\Edge[](v3)(v6)
\end{tikzpicture}}}
\qquad
\subfigure[$G_3$]{
\scalebox{\c}{
\begin{tikzpicture}
\GraphInit[vstyle=Classic]
\Vertex[Lpos=180,L=\hbox{\Huge $v_5$},x=0.0cm,y=3.1235cm]{v4}
\Vertex[L=\hbox{\Huge $v_2$},x=5.0cm,y=3.0618cm]{v1}
\Vertex[Lpos=180,L=\hbox{\Huge $v_4$},x=0.942cm,y=0.0422cm]{v3}
\Vertex[Lpos=90,L=\hbox{\Huge $v_1$},x=2.5253cm,y=5.0cm]{v0}
\Vertex[L=\hbox{\Huge $v_3$},x=4.0006cm,y=0.0cm]{v2}
\Vertex[L=\hbox{\Huge $v_6$},x=3cm,y=1cm]{v5}
\Vertex[Lpos=90,L=\hbox{\Huge $v_7$},x=2cm,y=2.3317cm]{v6}
\Edge[](v0)(v1)
\Edge[](v0)(v4)
\Edge[](v1)(v2)
\Edge[](v1)(v5)
\Edge[](v1)(v6)
\Edge[](v2)(v3)
\Edge[](v2)(v5)
\Edge[](v3)(v4)
\Edge[](v3)(v6)
\Edge[](v4)(v6)
\Edge[](v5)(v6)
\end{tikzpicture}}}
\qquad
\subfigure[$G_4$]{
\scalebox{\c}{
\begin{tikzpicture}
\GraphInit[vstyle=Classic]
\Vertex[Lpos=180,L=\hbox{\Huge $v_5$},x=0.0cm,y=3.1235cm]{v4}
\Vertex[L=\hbox{\Huge $v_2$},x=5.0cm,y=3.0618cm]{v1}
\Vertex[Lpos=180,L=\hbox{\Huge $v_4$},x=0.942cm,y=0.0422cm]{v3}
\Vertex[Lpos=90,L=\hbox{\Huge $v_1$},x=2.5253cm,y=5.0cm]{v0}
\Vertex[L=\hbox{\Huge $v_3$},x=4.0006cm,y=0.0cm]{v2}
\Vertex[L=\hbox{\Huge $v_6$},x=3cm,y=1cm]{v5}
\Vertex[Lpos=90,L=\hbox{\Huge $v_7$},x=2cm,y=2.3317cm]{v6}
\Edge[](v0)(v1)
\Edge[](v0)(v4)
\Edge[](v1)(v2)
\Edge[](v1)(v5)
\Edge[](v2)(v3)
\Edge[](v2)(v5)
\Edge[](v3)(v4)
\Edge[](v3)(v6)
\Edge[](v4)(v6)
\Edge[](v5)(v6)
\end{tikzpicture}}}
\qquad
\subfigure[$G_5$]{
\scalebox{\c}{
\begin{tikzpicture}
\GraphInit[vstyle=Classic]
\Vertex[Lpos=180,L=\hbox{\Huge $v_5$},x=0.0cm,y=3.1235cm]{v4}
\Vertex[L=\hbox{\Huge $v_2$},x=5.0cm,y=3.0618cm]{v1}
\Vertex[Lpos=180,L=\hbox{\Huge $v_4$},x=0.942cm,y=0.0422cm]{v3}
\Vertex[Lpos=90,L=\hbox{\Huge $v_1$},x=2.5253cm,y=5.0cm]{v0}
\Vertex[L=\hbox{\Huge $v_3$},x=4.0006cm,y=0.0cm]{v2}
\Vertex[Lpos=-35,L=\hbox{\Huge $v_6$},x=3cm,y=3.5cm]{v5}
\Vertex[Lpos=180,L=\hbox{\Huge $v_7$},x=2cm,y=2cm]{v6}
\Edge[](v0)(v1)
\Edge[](v0)(v4)
\Edge[](v0)(v6)
\Edge[](v1)(v2)
\Edge[](v1)(v5)
\Edge[](v2)(v3)
\Edge[](v2)(v5)
\Edge[](v2)(v6)
\Edge[](v3)(v4)
\Edge[](v3)(v6)
\Edge[](v4)(v5)
\end{tikzpicture}}}
\qquad
\subfigure[$G_6$]{
\scalebox{\c}{
\begin{tikzpicture}
\GraphInit[vstyle=Classic]
\Vertex[Lpos=180,L=\Vertex[Lpos=180,L=\hbox{\Huge $v_5$},x=0.0cm,y=3.1235cm]{v4}
\Vertex[L=\hbox{\Huge $v_2$},x=5.0cm,y=3.0618cm]{v1}
\Vertex[Lpos=180,L=\hbox{\Huge $v_4$},x=0.942cm,y=0.0422cm]{v3}
\Vertex[Lpos=90,L=\hbox{\Huge $v_1$},x=2.5253cm,y=5.0cm]{v0}
\Vertex[L=\hbox{\Huge $v_3$},x=4.0006cm,y=0.0cm]{v2}
\Vertex[L=\hbox{\Huge $v_6$},x=3cm,y=1cm]{v5}
\Vertex[Lpos=100,L=\hbox{\Huge $v_7$},x=2cm,y=2.3317cm]{v6}
\Vertex[Lpos=-65,L=\hbox{\Huge $v_8$},x=2.7107cm,y=4cm]{v7}
\Edge[](v0)(v1)
\Edge[](v0)(v4)
\Edge[](v1)(v2)
\Edge[](v1)(v5)
\Edge[](v1)(v7)
\Edge[](v2)(v3)
\Edge[](v2)(v5)
\Edge[](v3)(v4)
\Edge[](v3)(v6)
\Edge[](v3)(v7)
\Edge[](v4)(v6)
\Edge[](v4)(v7)
\Edge[](v5)(v6)
\Edge[](v5)(v7)
\end{tikzpicture}}}
\qquad
\subfigure[$G_7$]{
\scalebox{\c}{
\begin{tikzpicture}
\GraphInit[vstyle=Classic]
\Vertex[Lpos=180,L=\hbox{\Huge $v_5$},x=0.0cm,y=3.1235cm]{v4}
\Vertex[L=\hbox{\Huge $v_2$},x=5.0cm,y=3.0618cm]{v1}
\Vertex[Lpos=180,L=\hbox{\Huge $v_4$},x=0.942cm,y=0.0422cm]{v3}
\Vertex[Lpos=90,L=\hbox{\Huge $v_1$},x=2.5253cm,y=5.0cm]{v0}
\Vertex[L=\hbox{\Huge $v_3$},x=4.0006cm,y=0.0cm]{v2}
\Vertex[Lpos=90,L=\hbox{\Huge $v_6$},x=1.3cm,y=2.6cm]{v5}
\Vertex[Lpos=95,L=\hbox{\Huge $v_7$},x=4cm,y=1.5cm]{v6}
\Vertex[Lpos=180,L=\hbox{\Huge $v_8$},x=1.9cm,y=1.3cm]{v7}
\Edge[](v0)(v1)
\Edge[](v0)(v4)
\Edge[](v0)(v7)
\Edge[](v1)(v2)
\Edge[](v1)(v5)
\Edge[](v1)(v6)
\Edge[](v2)(v3)
\Edge[](v2)(v5)
\Edge[](v2)(v6)
\Edge[](v2)(v7)
\Edge[](v3)(v4)
\Edge[](v3)(v7)
\Edge[](v4)(v5)
\Edge[](v5)(v6)
\end{tikzpicture}}}
\qquad
\subfigure[$G_8$]{
\scalebox{\c}{
\begin{tikzpicture}
\GraphInit[vstyle=Classic]
\Vertex[Lpos=180,L=\hbox{\Huge $v_5$},x=0.0cm,y=3.1235cm]{v4}
\Vertex[L=\hbox{\Huge $v_2$},x=5.0cm,y=3.0618cm]{v1}
\Vertex[Lpos=180,L=\hbox{\Huge $v_4$},x=0.942cm,y=0.0422cm]{v3}
\Vertex[Lpos=90,L=\hbox{\Huge $v_1$},x=2.5253cm,y=5.0cm]{v0}
\Vertex[L=\hbox{\Huge $v_3$},x=4.0006cm,y=0.0cm]{v2}
\Vertex[Lpos=-15,L=\hbox{\Huge $v_6$},x=3cm,y=2cm]{v5}
\Vertex[Lpos=45,L=\hbox{\Huge $v_7$},x=3cm,y=3.0618cm]{v6}
\Vertex[Lpos=88,L=\hbox{\Huge $v_8$},x=1cm,y=2.5cm]{v7}
\Edge[](v0)(v1)
\Edge[](v0)(v4)
\Edge[](v0)(v6)
\Edge[](v1)(v2)
\Edge[](v1)(v5)
\Edge[](v1)(v6)
\Edge[](v2)(v3)
\Edge[](v2)(v5)
\Edge[](v3)(v4)
\Edge[](v3)(v6)
\Edge[](v3)(v7)
\Edge[](v4)(v7)
\Edge[](v5)(v7)
\Edge[](v6)(v7)
\end{tikzpicture}}}
\qquad
\subfigure[$G_9$]{
\scalebox{\c}{
\begin{tikzpicture}
\GraphInit[vstyle=Classic]
\Vertex[Lpos=180,L=\hbox{\Huge $v_5$},x=0.0cm,y=3.1235cm]{v4}
\Vertex[L=\hbox{\Huge $v_2$},x=5.0cm,y=3.0618cm]{v1}
\Vertex[Lpos=180,L=\hbox{\Huge $v_4$},x=0.942cm,y=0.0422cm]{v3}
\Vertex[Lpos=90,L=\hbox{\Huge $v_1$},x=2.5253cm,y=5.0cm]{v0}
\Vertex[L=\hbox{\Huge $v_3$},x=4.0006cm,y=0.0cm]{v2}
\Vertex[L=\hbox{\Huge $v_6$},x=2.5cm,y=1cm]{v5}
\Vertex[Lpos=180,L=\hbox{\Huge $v_7$},x=3.5cm,y=3.0023cm]{v6}
\Vertex[L=\hbox{\Huge $v_8$},x=1cm,y=3cm]{v7}
\Edge[](v0)(v1)
\Edge[](v0)(v4)
\Edge[](v0)(v6)
\Edge[](v1)(v2)
\Edge[](v1)(v5)
\Edge[](v2)(v3)
\Edge[](v2)(v5)
\Edge[](v2)(v6)
\Edge[](v3)(v4)
\Edge[](v3)(v6)
\Edge[](v3)(v7)
\Edge[](v4)(v5)
\Edge[](v4)(v7)
\Edge[](v5)(v7)
\end{tikzpicture}}}
\qquad
\subfigure[$G_{10}$]{
\scalebox{\c}{
\begin{tikzpicture}
\GraphInit[vstyle=Classic]
\Vertex[Lpos=180,L=\hbox{\Huge $v_5$},x=0.0cm,y=3.1235cm]{v4}
\Vertex[L=\hbox{\Huge $v_2$},x=5.0cm,y=3.0618cm]{v1}
\Vertex[Lpos=180,L=\hbox{\Huge $v_4$},x=0.942cm,y=0.0422cm]{v3}
\Vertex[Lpos=90,L=\hbox{\Huge $v_1$},x=2.5253cm,y=5.0cm]{v0}
\Vertex[L=\hbox{\Huge $v_3$},x=4.0006cm,y=0.0cm]{v2}
\Vertex[,L=\hbox{\Huge $v_6$},x=5.5cm,y=1.5cm]{v5}
\Vertex[Lpos=55,L=\hbox{\Huge $v_7$},x=4.5cm,y=5.5cm]{v6}
\Vertex[Lpos=125,L=\hbox{\Huge $v_8$},x=0.5cm,y=5.5cm]{v7}
\Vertex[Lpos=180,L=\hbox{\Huge $v_9$},x=-0.5cm,y=1.5cm]{v8}
\Edge[](v0)(v1)
\Edge[](v0)(v4)
\Edge[](v0)(v6)
\Edge[](v0)(v7)
\Edge[](v1)(v2)
\Edge[](v1)(v5)
\Edge[](v1)(v6)
\Edge[](v2)(v3)
\Edge[](v2)(v5)
\Edge[](v2)(v7)
\Edge[](v3)(v4)
\Edge[](v3)(v6)
\Edge[](v3)(v8)
\Edge[](v4)(v7)
\Edge[](v4)(v8)
\Edge[](v5)(v7)
\Edge[](v5)(v8)
\Edge[](v6)(v8)
\end{tikzpicture}}}
\caption{Special (gem, co-gem)-free graphs used in the structural characterization in \cite{KarthickMaffraygemcogem}.}\label{fig:Gis}
\end{figure}

Let $\calH$ be the class of (gem, co-gem)-free graphs $G$ such that $V(G)$ can be partitioned into six non-empty sets, $A_1,\ldots,A_6$ such that:
\begin{itemize}
\item $A_1$ is complete to $A_2\cup A_5$ and anti-complete to $A_3\cup A_4\cup A_6$,
\item $A_2$ is complete to $A_1\cup A_3\cup A_6$ and anti-complete to $A_4\cup A_5$,
\item $A_3$ is complete to $A_2\cup A_6\cup A_4$ and anti-complete to $A_1\cup A_5$, and
\item $A_4$ is complete to $A_3\cup A_5$ and anti-complete to $A_1\cup A_2\cup A_6$.
\end{itemize}
\noindent Note that $A_i$ is $P_4$-free for all $i=1,2,\ldots, 6$ and the adjacencies between $A_5$ and $A_6$ are unspecified but restricted by the fact that $G$ is (gem, co-gem)-free. A depiction of a graph in $\calH$ is given in Figure~\ref{fig:H}. We can now state the structural characterization.

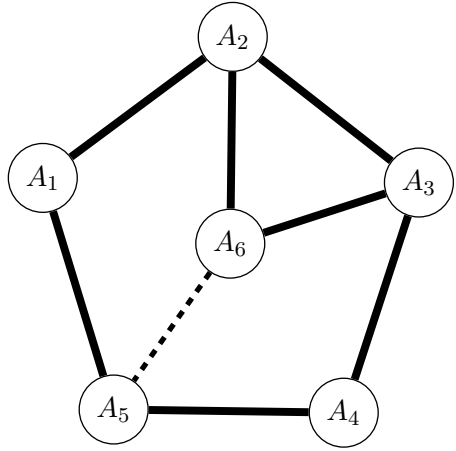
\begin{figure}
\centering
\begin{tikzpicture}
\GraphInit[vstyle=Normal]
\Vertex[L=\hbox{$A_1$},x=0.0cm,y=3.1235cm]{A1}
\Vertex[L=\hbox{$A_3$},x=5.0cm,y=3.0618cm]{A3}
\Vertex[L=\hbox{$A_5$},x=0.942cm,y=0.0422cm]{A5}
\Vertex[L=\hbox{$A_2$},x=2.5253cm,y=5.0cm]{A2}
\Vertex[L=\hbox{$A_4$},x=4.0006cm,y=0.0cm]{A4}
\Vertex[L=\hbox{$A_6$},x=2.4923cm,y=2.2487cm]{A6}
\SetUpEdge[lw=3pt]

\Edge[](A1)(A2)
\Edge[](A1)(A5)
\Edge[](A2)(A3)
\Edge[](A2)(A6)
\Edge[](A3)(A4)
\Edge[](A3)(A6)
\Edge[](A4)(A5)
\SetUpEdge[lw=2pt]
\draw [line width = 2pt, dashed] (A6)--(A5);
%
\end{tikzpicture}
\caption{General form of a graph in $\calH$ where the thick black lines denote sets which are complete to each other and the dashed line denotes unspecified adjacency.}\label{fig:H}
\end{figure}

\begin{theorem}[\cite{KarthickMaffraygemcogem}]\label{thm:gemcogemstructure}
If $G$ is (gem,co-gem)-free, then either $G$ is perfect, or $G\in\mathcal{G}_i$ for some $i\in\{1,2,\ldots, 10\}$, or $G\in \calH$.
\end{theorem}

\begin{lemma}[\cite{Dhaliwal2017}]\label{lem:module}
If $G$ is $k$-vertex-critical with a non-trivial module $M$, then $M$ is $m$-vertex-critical for some $m<k$.
\end{lemma}
%
%

\begin{lemma}[\cite{KarthickMaffraygemcogem}]\label{lem:chromaticnumberofC5exp}
If $G\in \mathcal{G}_{1}^{\ast}$ has order $n$, then $\chi(G)=\max\left(\omega(G),\left\lceil\frac{n}{2}\right\rceil\right)$.
\end{lemma}

\begin{lemma}[\cite{KarthickMaffraygemcogem}]\label{lem:vvery goodstablesetG2G10}
Let $i\in\{2,3,\ldots, 10\}$ and $G\in \mathcal{G}_i^{\ast}$. If $x_1\in A_1$, $x_4\in A_4$, and $x_6\in A_6$, then $\{x_1,x_4,x_6\}$ is a very good stable set in $G$.
\end{lemma}

The following observation is easy to establish by the fact that every induced $C_5$ in $G_i$ for each $i=2,\ldots, 10$ meets exactly two vertices in the set $\{v_1,v_4,v_6\}$ (where $v_i$ refers to the vertex labels of each graph as shown in Figure~\ref{fig:Gis}).
\begin{obs}\label{obs:vvery goodsetintat2}
Let $i\in \{\blue{2},\ldots,10\}$. If $G\in \mathcal{G}_i^{\blue{\ast}}$ 
and $C\subset V(G)$ such that \blue{$C$ induces a clique expansion of $C_5$, then $C\cap A_i\neq \emptyset$ for exactly two $i\in \{1,4,6\}$.}  \hfill$\Box$
\end{obs}

Let $\Hstar$ be the set of graphs in $\calH$ where $A_i$ is a clique for $i=1,2,3,4$. Recall from Section~\ref{sec:definitions} that $S^j$ denotes the vertices in the set $S$ that are assigned color $j$. 

\begin{lemma}\label{lem:singlecolorchange}
Let $\{i,i'\}=\{1,4\}$. If $G$ is a graph in $\Hstar$ with $\chi(G)=k$ and $j,j'\in \{1,2\ldots,k\}$ such that $j'\in c(A_6)\setminus c(A_i)$ and  $j\in c(A_i)\setminus c(A_6)$,
then a new $k$-coloring of $G$ can be obtained such that $j\not\in c(A_i)$ and $j'\in c(A_i)\cap c(A_6)$ by performing one of the following:
\begin{itemize}
\item[i)] Changing the color of a vertex in $A_i$ from $j$ to $j'$, or
\item[ii)] Swapping the colors of the vertices in $A_i^j$ and $A_5^{j'}$ and if there is a 
vertex in $A_{i'}$ with color $j$, then change the color of this vertex to $j'$.
\end{itemize}
\end{lemma}
\begin{proof}
By symmetry we may assume without loss of generality that $j'\in c(A_6)\setminus c(A_1)$ and $j\in c(A_1)\setminus c(A_6)$. If $j'\not\in c(A_5)$, then simply change the color $j$ in $A_1$ to $j'$ (operation \textit{i)}). Note that $j'\not\in c(A_2)$ since $A_6$ is complete to $A_6$ and $j'\in c(A_6)$ so this is a valid $k$-coloring. If $j'\in c(A_5)$, then swap the colors of $A_5^{j'}$ and $A_1^{j}$. Since $j\not\in c(A_6)$, this is a valid $k$-coloring if and only if $j\not\in c(A_4)$. Suppose $j\in c(A_4)$. Since $j'$ is no longer in $c(A_5)$ and is not in $c(A_3)$ since $A_6$ is complete to $A_3$, we can simply change the color $j$ in $A_4$ to $j'$ and the result is a valid $k$-coloring (operation \textit{ii)}).
\end{proof}

\begin{lemma}\label{lem:nocriticalinH}
There are no $k$-vertex-critical graphs in $\calH$.
\end{lemma}
\begin{proof}
Suppose by way of contradiction that $G$ is a $k$-vertex-critical graph in $\calH$. We will first establish a number of claims.

\begin{claim}\label{cla:GcriticalinHthenGinHstar}
$G\in \Hstar$.
\end{claim} 
\begin{proof}[Proof of Claim~\ref{cla:GcriticalinHthenGinHstar}]
By definition of ${\cal H}$, $A_i$ is a module for $i=1,2,3,4$, and therefore must be critical by Lemma~\ref{lem:module}. Since each $A_i$ must be $P_4$-free, it follows that $A_i$ is a clique for $i=1,2,3,4$. Thus, $G\in\Hstar$.
\end{proof}

\noindent For the rest of the proof, suppose without loss of generality that $|A_4|\le |A_1|$.

\begin{claim}\label{cla:chiA6<=cA1andcA4}
$\chi(G[A_6])< |A_4|$.
\end{claim}
\begin{proof}[Proof of Claim~\ref{cla:chiA6<=cA1andcA4}]
If there exists a $k$-coloring of $G$ such that $|c(A_6)|< |A_4|$, then we are done. So suppose for every $k$-coloring of $G$, $|c(A_6)|\ge|A_4|$.  Fix a $k$-coloring of $G$ such that $x\in A_4$ is the only vertex colored $k$. By successive applications of Lemma~\ref{lem:singlecolorchange}, we can get a new $k$-coloring of $G$ where $x$ is still the only vertex colored $k$, $c(A_4\setminus \{x\})\subset c(A_6)$ and $|c(A_1)\cap c(A_6)|>|A_4|-1$. Fix such a $k$-coloring of $G$. So there is a color $j\in (c(A_1)\cap c(A_6))\setminus c(A_4)$. Since $A_1$ is complete to $A_5$ and $A_6$ is complete to $A_3$, it follows that $j\not\in c(A_3)\cup c(A_5)$. Therefore, we can change the color of $x$ from $k$ to $j$ to produce a $(k-1)$-coloring of $G$, a contradiction.
\end{proof}

\noindent Let $G'$ be the graph defined by $V(G')=V(G)$ and $E(G')=E(G)\cup \{a_5a_6:a_5\in A_5\text{ and } a_6\in A_6\}$.

\begin{claim}\label{cla:G'alsokcrit}
$G'$ is $k$-vertex-critical and $A_5$ and $A_6$ are both cliques.
\end{claim}
\begin{proof}[Proof of Claim~\ref{cla:G'alsokcrit}]
Fix a $k$-coloring of $G$ such that $x$ is the only vertex colored $k$. We now show that no matter which set $A_i$ that $x$ is in, we can obtain a $k$-coloring of $G$ where $x$ is still the only vertex colored $k$ and where $c(A_5)\cap c(A_6)=\emptyset$.\\

\noindent \textit{Case 1:} $x\in A_2\cup A_3\cup A_5$.

If $|c(A_6)|\le |A_4|$, then by successive applications of Lemma~\ref{lem:singlecolorchange} we can obtain a new $k$-coloring where $x$ is still the only vertex colored $k$ and $c(A_6)\subseteq c(A_4)$. Therefore $c(A_5)\cap c(A_6)\subseteq c(A_5)\cap c(A_4)=\emptyset$ since $A_4$ is complete to $A_5$. Thus, we may assume $|c(A_6)|>|A_4|$. Then by successive applications of Lemma~\ref{lem:singlecolorchange} we can obtain a new $k$-coloring where $x$ is still the only vertex colored $k$ and $c(A_4)\subset c(A_6)$. From  Claim~\ref{cla:GcriticalinHthenGinHstar} and Claim~\ref{cla:chiA6<=cA1andcA4},  $|c(A_4)|=|A_4|> \chi(G[A_6])$. Therefore, we can obtain a new $k$-coloring of $G$ by coloring $A_6$ with any $\chi(G[A_6])$ colors from $c(A_4)$. Note that this will still be a valid $k$ coloring as $c(A_4)\cap c(A_5)=\emptyset$ and $c(A_4)\cap c(A_3)=\emptyset$ since $A_4$ is complete to both sets and $c(A_4)\cap c(A_2)=\emptyset$ since $A_6$ is complete to $A_2$ and $c(A_4)\subset c(A_6)$.  In such a $k$-coloring, we have $c(A_5)\cap c(A_6)=\emptyset$ since $c(A_4)\cap c(A_5)=\emptyset$. \\

\noindent\textit{Case 2:} $x\in A_6$.

If $|c(A_6)|\le |A_4|$, then we may successively apply Lemma~\ref{lem:singlecolorchange} to obtain a new $k$-coloring of $G$ where $x$ remains the only vertex with color $k$ and $c(A_6)\setminus\{k\}\subseteq c(A_4)$. Since $k\not\in c(A_5)$, we now have $c(A_6)\cap c(A_5)\subseteq c(A_4)\cap c(A_5)=\emptyset$. Thus, we may assume $|c(A_6)|>|A_4|$. We may successively apply Lemma~\ref{lem:singlecolorchange} to obtain a new $k$-coloring of $G$ where $x$ remains the only vertex with color $k$ and $c(A_4)\subseteq c(A_6)\setminus\{k\}$. By Claim~\ref{cla:GcriticalinHthenGinHstar} and Claim~\ref{cla:chiA6<=cA1andcA4}, it follows that $|c(A_4)\cap c(A_6)|> \chi(G[A_6])$. Therefore, we can obtain a new $k$-coloring of $G$ by coloring $A_6$ with any $\chi(G[A_6])$ colors from $c(A_4)$ as in Case 1. In such a $k$-coloring, we have $c(A_5)\cap c(A_6)=\emptyset$ since $c(A_4)\cap c(A_5)=\emptyset$. \\

\noindent \textit{Case 3:} $x\in A_1\cup A_4$.

Let $\{i,i'\}=\{1,4\}$ and suppose that $x\in A_i$. If $|c(A_6)|\le |A_{i'}|$, then we may successively apply Lemma~\ref{lem:singlecolorchange} to obtain a new $k$-coloring of $G$ where $x$ remains the only vertex with color $k$ and $c(A_6)\subseteq c(A_{i'})$. In this $k$-coloring, $c(A_5)\cap c(A_6)=\emptyset$. Thus, we may assume $|c(A_6)|>|A_{i'}|$. We may successively apply Lemma~\ref{lem:singlecolorchange} to obtain a new $k$-coloring of $G$ where $x$ remains the only vertex with color $k$ and $c(A_{i'})\subseteq c(A_6)$. By Claim~\ref{cla:GcriticalinHthenGinHstar} and Claim~\ref{cla:chiA6<=cA1andcA4} and since $|A_{i'}|\ge|A_4|$, it follows that $|c(A_{i'})\cap c(A_6)|> \chi(A_6)$. Therefore, we can obtain a new $k$-coloring of $G$ by coloring $A_6$ with any $\chi(A_6)$ colors from $c(A_{i'})$. In such a $k$-coloring, we have $c(A_5)\cap c(A_6)=\emptyset$ since $c(A_{i'})\cap c(A_5)=\emptyset$.\\


In each case, it follows that the $k$-coloring of $G$ is also a $k$-coloring for $G'$. Since $G'$ is a $k$-colorable supergraph of $G$ and $\chi(G)=k$, we must have $\chi(G')=k$. Further, since $x$ is the only vertex colored $k$, $G'-x$ is $(k-1)$-colorable. Since this holds for any $x\in V(G')$, $G'$ is $k$-vertex-critical. Hence, $A_5$ and $A_6$ are modules in $G'$ and therefore critical by Lemma~\ref{lem:module}. Since $A_5$ and $A_6$ are $P_4$-free and critical, they must be cliques. 
\end{proof}

\red{Now, we continue the proof of Lemma~\ref{lem:nocriticalinH}. From Claim~\ref{cla:G'alsokcrit}, we may assume $G'$ is $k$-vertex-critical.}
Fix a $k$-coloring of $G'$ such that $x\in A_6$ is the only vertex with color $k$. 
If there exists $c_1\in c(A_1)\setminus (c(A_6)\cup c(A_3))$ or $c_1\in c(A_4)\setminus (c(A_6)\cup c(A_2))$, then $c_1\not\in c(A_5)$ so we can change the color of $x$ from $k$ to $c_1$ and obtain a $(k-1)$-coloring of $G'$, a contradiction. Therefore, we have $c(A_1)\subseteq c(A_6)\cup c(A_3)$ and $c(A_4)\subseteq c(A_6)\cup c(A_2)$. Note that $|c(A_6)|=|A_6|$ from Claim~\ref{cla:G'alsokcrit} and since $|A_6|< |A_4|$ from Claim~\ref{cla:chiA6<=cA1andcA4}, it follows that there is a color $c_2\in c(A_4)\cap c(A_2)$. Further, if $c(A_5)\subseteq c(A_2)\cup c(A_3)$, then $k=|A_6|+|A_2|+|A_3|=\chi(G'-(A_1\cup A_4\cup A_5))$, contradicting $G'$ being $k$-vertex-critical. Therefore, there is a color $c_3\in c(A_5)\setminus (c(A_2)\cup c(A_3))$. Thus, we can swap the colors of the vertices in the sets $A_4^{c_2}$ and $A_5^{c_3}$ to obtain a new $k$-coloring of $G'$. Finally, we can now change the color of $x$ to $c_3$ to obtain a $(k-1)$-coloring of $G'$, a contradiction. \red{Thus, the Lemma holds.}
\end{proof}

\begin{lemma}\label{lem:nocritinG2G10}
There are no vertex-critical graphs in $\mathcal{G}_2\cup \mathcal{G}_3\cup \cdots \cup \mathcal{G}_{10}$.
\end{lemma}
\begin{proof}
Let $\mathcal{S}=\mathcal{G}_2\cup \mathcal{G}_3\cup \cdots \cup \mathcal{G}_{10}$ and $\mathcal{S}^{\ast}=\mathcal{G}_2^{\ast}\cup \mathcal{G}_3^{\ast}\cup \cdots \cup \mathcal{G}_{10}^{\ast}$. The proof proceeds by induction on $k$.  For $k=1,2,3$ the result is clear. For some $k>3$, suppose there are no $k'$-vertex-critical graphs in $\mathcal{S}$ for all $k'<k$ and suppose $G$ is a $k$-vertex-critical graph in $\mathcal{S}$. By Lemma~\ref{lem:module} and since $G$ is gem-free, it follows that $G\in\mathcal{S}^{\ast}$. Note that $\omega(G)\le k-1$ since $G$ is critical and not $K_k$. 

Let $x_1\in A_1$, $x_4\in A_4$, $x_6\in A_6$ and $S=\{x_1,x_4,x_6\}$. Then by Lemma~\ref{lem:vvery goodstablesetG2G10}, $S$ is a very good stable set in $G$. Since $S$ is a  very good stable set and $G$ is $k$-vertex-critical, $G-S$ must be $(k-1)$-chromatic and $\omega(G-S)<\omega(G)\le k-1$. \blue{Since $G-S$ is (gem, co-gem)-free, it is $(C_{2k+1},\overline{C_{2k+1}})$-free for all $k\ge 3$, so if $G-S$ is $C_5$-free, then it must be perfect by the Strong Perfect Graph Theorem~\cite{Chudnovsky2006}. This would then imply that $G-S$ is $K_{k-1}$, contradicting $\omega(G)<k-1$.} 
%
%

\blue{Therefore assume $G-S$ contains an induced $C_5$. Since $\chi(G-S)=k-1$, $G-S$ must
contain an induced $(k-1)$-vertex-critical subgraph $F$. By the same reasoning as for $G-S$, $F$ must also contain an induced $C_5$. Thus, by Theorem~\ref{thm:gemcogemstructure}, $F\in \mathcal{G}_i$ for some $i\in\{1,\ldots,10\}$ or $F\in \mathcal{H}$. But since $F$ is $(k-1)$-vertex-critical, Lemma~\ref{lem:nocriticalinH} gives $F\not\in\calH$ and the inductive hypothesis gives that $F\not\in \mathcal{G}_i$ for all $i\in \{2,\ldots,10\}$. Thus, $F\in \mathcal{G}_1$ and further, by Lemma~\ref{lem:module} it follows that $F\in \mathcal{G}_1^{\ast}$.}  
Since $\omega(F)\le \omega(G-S)<k-1$, Lemma~\ref{lem:chromaticnumberofC5exp} gives $k-1=\left\lceil\frac{|V(F)|}{2}\right\rceil$.
By Observation~\ref{obs:vvery goodsetintat2}, \blue{$V(F)\cap A_i\neq \emptyset$ for exactly two $i\in \{1,4,6\}$. Without loss of generality, say $V(F)\cap A_1\neq \emptyset$ and $V(F)\cap A_4\neq \emptyset$. Thus, } $G[V(F)\cup\{\blue{x_1,x_4}\}]\in\mathcal{G}_1^{\ast}$. However, we now have $G[V(F)\cup\{\blue{x_1,x_4}\}]$ being a proper induced subgraph of $G$ and, by Lemma~\ref{lem:chromaticnumberofC5exp},
\begin{align*}
\chi(G[V(F)\cup\{\blue{x_1,x_4}\}])&\le\left\lceil\frac{\left\vert V(F)\cup\{\blue{x_1,x_4}\}\right\vert}{2}\right\rceil\\
&= \left\lceil\frac{|V(F)|}{2}\right\rceil+1\\
&=k.
\end{align*}
This contradicts $G$ being $k$-vertex-critical. Therefore there are no critical graphs in $\mathcal{S}$.
\end{proof}

\begin{theorem}
For all $k\ge 1$, there are only finitely many $k$-vertex-critical (gem, co-gem)-free graphs. 
\end{theorem}
\begin{proof}
Let $G$ be a $k$-critcal (gem, co-gem)-free graph. By Theorem~\ref{thm:gemcogemstructure}, $G$ is either perfect, or in $\mathcal{G}_i$ for $i\in\{1,2,\ldots, 10\}$ or in $\calH$. If $G$ is perfect, then $G=K_k$. By Lemma~\ref{lem:nocriticalinH}, $G\not\in \calH$.  If $G\in \mathcal{G}_i$ for some $i\in\{1,2,\ldots, 10\}$, then, by Lemma~\ref{lem:nocritinG2G10}, $i=1$. By Lemma~\ref{lem:module}, $G$ must be a clique expansion of $G_1$. Using liberal upper bounds, each clique can have at most $k-2$ vertices and $|V(G_1)|=5$, so $|V(G)|<5(k-2)$. Therefore, there are only finitely many $k$-vertex-critical (gem,co-gem)-free graphs.
\end{proof}

\subsection{Exact number of $k$-vertex-critical (gem, co-gem)-free graphs}\label{subsec:exhaustivegemcogemgeneration}
Lemmas~\ref{lem:nocriticalinH}~and~\ref{lem:nocritinG2G10} give the stronger result that every $k$-vertex-critical (gem, co-gem)-free graph is either complete or in $\mathcal{G}_1^{\ast}$, i.e., a clique-expansion of $C_5$. The following theorem provides further restrictions on the structure of vertex-critical graphs in $\mathcal{G}_1^\ast$.

\begin{theorem}[\cite{Dhaliwal2017}]\label{thm:critC5cliqueexpcritrea} 
%
Let $k\ge 3$. A graph $G\in\mathcal{G}_1^{\ast}$ is $k$-vertex-critical if and only if $G[A_i]\cong K_{k_i}$ such that $\sum_{i=1}^{5}k_i=2k-1$ and for each $i$ mod $5$, $k_i+k_{i+1}\le k-1$.
\end{theorem}

\begin{table}[h]
\begin{center}
\begin{tabular}{c|c|c|c|c|c|c|c|c|c|c|c|c|c|c|c|c}
$k$ & 1 & 2 & 3 & 4 & 5 & 6 & 7 & 8 & 9 & 10 & 11 & 12 & 13 & 14 & 15 & 16 \\ \hline
num$(k)$ & 1&1&2&2&4&6&11&17&27&39&58&80&112&148&197&253
\end{tabular}
\end{center}
\caption{num$(k)$ denotes the number of $k$-vertex-critical (gem, co-gem)-free graphs.}\label{tab:critnums}
\end{table}

This allows for efficient generation of all $k$-vertex-critical (gem,co-gem)-free graphs for all $k\le 16$ by simply computing all integer compositions of $2k-1$ such that adjacent parts sum to at most $k-1$ and the symmetries of $C_5$ are accounted for. The results of this are summarized in Table~\ref{tab:critnums} and Figure~\ref{fig:6critgemcogem} shows all $6$-vertex-critical (gem, co-gem)-free graphs. The graphs in graph6 format as well as Sagemath files used to generate them are available at~\cite{graphfiles}. 
%

\setcounter{subfigure}{0}
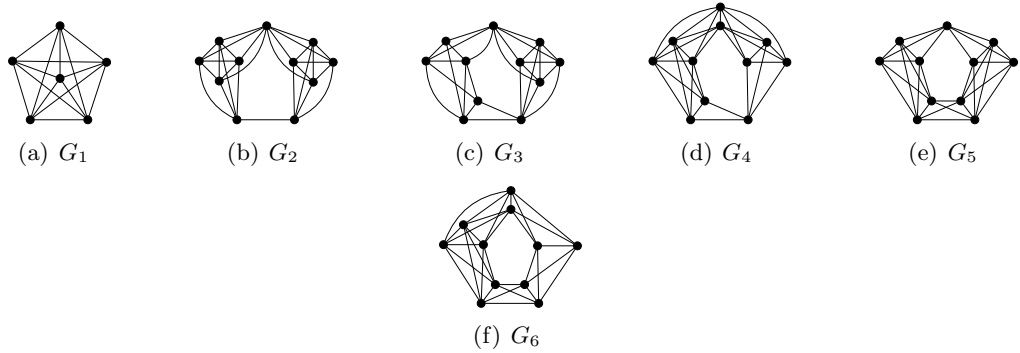
\begin{figure}[htb]
\def\c{0.25}
\def\r{1}
\centering
\subfigure[$G_1$]{
\scalebox{\c}{
\begin{tikzpicture}
\GraphInit[vstyle=Classic]
\Vertex[Lpos=180,L=\hbox{},x=\r*0.0cm,y=\r*3.1235cm]{v0}
\Vertex[L=\hbox{},x=\r*5.0cm,y=\r*3.0618cm]{v1}
\Vertex[Lpos=180,L=\hbox{},x=\r*0.942cm,y=\r*0.0cm]{v2}
\Vertex[Lpos=90,L=\hbox{},x=\r*2.5253cm,y=\r*5.0cm]{v3}
\Vertex[L=\hbox{},x=\r*4.0006cm,y=\r*0.0cm]{v4}
\Vertex[L=\hbox{},x=\r*2.5253cm,y=\r*2.2cm]{v5}
%
%

%
\Edge[](v0)(v1)
\Edge[](v0)(v2)
\Edge[](v0)(v3)
\Edge[](v0)(v4)
\Edge[](v0)(v5)
\Edge[](v1)(v2)
\Edge[](v1)(v3)
\Edge[](v1)(v4)
\Edge[](v1)(v5)
\Edge[](v2)(v3)
\Edge[](v2)(v4)
\Edge[](v2)(v5)
\Edge[](v3)(v4)
\Edge[](v3)(v5)
\Edge[](v4)(v5)
\end{tikzpicture}}}
\qquad
\subfigure[$G_2$]{
\scalebox{\c}{
\begin{tikzpicture}
\GraphInit[vstyle=Classic]

\Vertex[Lpos=90,L=\hbox{},x=\r*2.5253cm,y=\r*5.0cm]{v1}
\Vertex[L=\hbox{},x=\r*5.0cm,y=\r*4.12246017177982cm]{v2}
\Vertex[L=\hbox{},x=\r*5.0cm,y=\r*2.00113982822018cm]{v21}
\Vertex[L=\hbox{},x=\r*3.93933982822018cm,y=\r*3.0618cm]{v22}
\Vertex[L=\hbox{},x=\r*6.06066017177982cm,y=\r*3.0618cm]{v23}
\Vertex[L=\hbox{},x=\r*4.0006cm,y=\r*0.0cm]{v3}
\Vertex[Lpos=180,L=\hbox{},x=\r*0.942cm,y=\r*0cm]{v4}
\Vertex[Lpos=180,L=\hbox{},x=\r*0.0cm,y=\r*4.18416017177982cm]{v5} 
\Vertex[L=\hbox{},x=\r*-1.0606601717798cm,y=\r*3.12350000000000cm]{v51}
\Vertex[L=\hbox{},x=\r*1.06066017177982cm,y=\r*3.12350000000000cm]{v52}
\Vertex[L=\hbox{},x=\r*0cm,y=\r*2.06283982822018cm]{v53}
%
\Edge[](v1)(v2)
\Edge[](v3)(v2)
\Edge[](v4)(v3)
\Edge[](v5)(v4)
\Edge[](v1)(v5)

\Edge[](v2)(v21)
\Edge[](v2)(v22)
\Edge[](v2)(v23)
\Edge[](v21)(v22)
\Edge[](v21)(v23)
\Edge[](v22)(v23)

\Edge[](v1)(v22)
\Edge[](v1)(v23)

\Edge[](v3)(v21)
\Edge[](v3)(v22)

\Edge[](v5)(v51)
\Edge[](v5)(v52)
\Edge[](v5)(v53)
\Edge[](v51)(v52)
\Edge[](v51)(v53)
\Edge[](v52)(v53)

\Edge[](v1)(v51)
\Edge[](v1)(v52)

\Edge[](v4)(v52)
\Edge[](v4)(v53)

\tikzset{EdgeStyle/.append style = {bend right}}
\Edge[](v1)(v21)
\Edge[](v3)(v23)

\tikzset{EdgeStyle/.append style = {bend left}}
\Edge[](v1)(v53)
\Edge[](v4)(v51)

\end{tikzpicture}}}
\qquad
\subfigure[$G_3$]{
\scalebox{\c}{
\begin{tikzpicture}
\GraphInit[vstyle=Classic]
\Vertex[Lpos=90,L=\hbox{},x=\r*2.5253cm,y=\r*5.0cm]{v1}
\Vertex[L=\hbox{},x=\r*5.0cm,y=\r*4.12246017177982cm]{v2}
\Vertex[L=\hbox{},x=\r*5.0cm,y=\r*2.00113982822018cm]{v21}
\Vertex[L=\hbox{},x=\r*3.93933982822018cm,y=\r*3.0618cm]{v22}
\Vertex[L=\hbox{},x=\r*6.06066017177982cm,y=\r*3.0618cm]{v23}
\Vertex[L=\hbox{},x=\r*4.0006cm,y=\r*0.0cm]{v3}
\Vertex[Lpos=180,L=\hbox{},x=\r*0.942cm,y=\r*0cm]{v4}
\Vertex[Lpos=180,L=\hbox{},x=\r*1.69650961064564cm,y=\r*1cm]{v41}
\Vertex[Lpos=180,L=\hbox{},x=\r*0.0cm,y=\r*4.18416017177982cm]{v5} 
\Vertex[L=\hbox{},x=\r*-1.0606601717798cm,y=\r*3.12350000000000cm]{v51}
\Vertex[L=\hbox{},x=\r*1.06066017177982cm,y=\r*3.12350000000000cm]{v52}
%
\Edge[](v1)(v2)
\Edge[](v3)(v2)
\Edge[](v4)(v3)
\Edge[](v5)(v4)
\Edge[](v1)(v5)

\Edge[](v2)(v21)
\Edge[](v2)(v22)
\Edge[](v2)(v23)
\Edge[](v21)(v22)
\Edge[](v21)(v23)
\Edge[](v22)(v23)

\Edge[](v1)(v22)
\Edge[](v1)(v23)

\Edge[](v3)(v21)
\Edge[](v3)(v22)

\Edge[](v4)(v41)

\Edge[](v5)(v51)
\Edge[](v5)(v52)
\Edge[](v51)(v52)

\Edge[](v1)(v51)

\Edge[](v4)(v52)

\Edge[](v41)(v51)
\Edge[](v41)(v52)

\Edge[](v41)(v3)

\tikzset{EdgeStyle/.append style = {bend right}}
\Edge[](v1)(v21)
\Edge[](v3)(v23)

\tikzset{EdgeStyle/.append style = {bend left}}
\Edge[](v1)(v52)
\Edge[](v4)(v51)
\end{tikzpicture}}}
\qquad
\subfigure[$G_4$]{
\scalebox{\c}{
\begin{tikzpicture}
\GraphInit[vstyle=Classic]
\Vertex[Lpos=90,L=\hbox{},x=\r*2.5253cm,y=\r*5.0cm]{v1}
\Vertex[Lpos=90,L=\hbox{},x=\r*2.5253cm,y=\r*6.0cm]{v11}
\Vertex[L=\hbox{},x=\r*5.0cm,y=\r*4.12246017177982cm]{v2}
\Vertex[L=\hbox{},x=\r*3.93933982822018cm,y=\r*3.0618cm]{v22}
\Vertex[L=\hbox{},x=\r*6.06066017177982cm,y=\r*3.0618cm]{v23}
\Vertex[L=\hbox{},x=\r*4.0006cm,y=\r*0.0cm]{v3}
\Vertex[Lpos=180,L=\hbox{},x=\r*0.942cm,y=\r*0cm]{v4}
\Vertex[Lpos=180,L=\hbox{},x=\r*1.69650961064564cm,y=\r*1cm]{v41}
\Vertex[Lpos=180,L=\hbox{},x=\r*0.0cm,y=\r*4.18416017177982cm]{v5} 
\Vertex[L=\hbox{},x=\r*-1.0606601717798cm,y=\r*3.12350000000000cm]{v51}
\Vertex[L=\hbox{},x=\r*1.06066017177982cm,y=\r*3.12350000000000cm]{v52}
%
\Edge[](v1)(v2)
\Edge[](v4)(v3)
\Edge[](v5)(v4)
\Edge[](v1)(v5)

\Edge[](v2)(v22)
\Edge[](v2)(v23)
\Edge[](v22)(v23)

\Edge[](v1)(v22)
\Edge[](v1)(v23)

\Edge[](v3)(v2)

\Edge[](v3)(v22)

\Edge[](v4)(v41)

\Edge[](v5)(v51)
\Edge[](v5)(v52)
\Edge[](v51)(v52)

\Edge[](v1)(v51)
\Edge[](v1)(v52)

\Edge[](v1)(v11)
\Edge[](v11)(v22)

\Edge[](v11)(v5)
\Edge[](v11)(v52)

\Edge[](v4)(v52)
\Edge[](v4)(v51)

\Edge[](v2)(v11)

\Edge[](v41)(v5)
\Edge[](v41)(v51)
\Edge[](v41)(v52)

\Edge[](v41)(v3)
\Edge[](v3)(v23)
\tikzset{EdgeStyle/.append style = {bend right}}
\Edge[](v23)(v11)
\Edge[](v11)(v51)

\tikzset{EdgeStyle/.append style = {bend left}}

\end{tikzpicture}}}
\qquad
\subfigure[$G_5$]{
\scalebox{\c}{
\begin{tikzpicture}
\GraphInit[vstyle=Classic]
\Vertex[Lpos=90,L=\hbox{},x=\r*2.5253cm,y=\r*5.0cm]{v1}
\Vertex[L=\hbox{},x=\r*5.0cm,y=\r*4.12246017177982cm]{v2}
\Vertex[L=\hbox{},x=\r*3.93933982822018cm,y=\r*3.0618cm]{v22}
\Vertex[L=\hbox{},x=\r*6.06066017177982cm,y=\r*3.0618cm]{v23}
\Vertex[L=\hbox{},x=\r*4.0006cm,y=\r*0.0cm]{v3}
\Vertex[L=\hbox{},x=\r*
3.25149038935436cm,y=\r*1cm]{v31}
\Vertex[Lpos=180,L=\hbox{},x=\r*0.942cm,y=\r*0cm]{v4}
\Vertex[Lpos=180,L=\hbox{},x=\r*1.69650961064564cm,y=\r*1cm]{v41}
\Vertex[Lpos=180,L=\hbox{},x=\r*0.0cm,y=\r*4.18416017177982cm]{v5} 
\Vertex[L=\hbox{},x=\r*-1.0606601717798cm,y=\r*3.12350000000000cm]{v51}
\Vertex[L=\hbox{},x=\r*1.06066017177982cm,y=\r*3.12350000000000cm]{v52}
%
\Edge[](v1)(v2)
\Edge[](v3)(v2)
\Edge[](v4)(v3)
\Edge[](v5)(v4)
\Edge[](v1)(v5)

\Edge[](v2)(v22)
\Edge[](v2)(v23)
\Edge[](v22)(v23)

\Edge[](v1)(v22)
\Edge[](v1)(v23)

\Edge[](v3)(v31)

\Edge[](v2)(v31)
\Edge[](v22)(v31)
\Edge[](v23)(v31)
\Edge[](v4)(v31)
\Edge[](v41)(v31)

\Edge[](v3)(v22)

\Edge[](v4)(v41)

\Edge[](v5)(v51)
\Edge[](v5)(v52)
\Edge[](v51)(v52)

\Edge[](v1)(v51)
\Edge[](v1)(v52)

\Edge[](v4)(v52)
\Edge[](v4)(v51)

\Edge[](v41)(v5)
\Edge[](v41)(v51)
\Edge[](v41)(v52)

\Edge[](v41)(v3)
\Edge[](v3)(v23)
\tikzset{EdgeStyle/.append style = {bend right}}

\tikzset{EdgeStyle/.append style = {bend left}}
\end{tikzpicture}}}
\qquad
\subfigure[$G_6$]{
\scalebox{\c}{
\begin{tikzpicture}
\GraphInit[vstyle=Classic]
\Vertex[Lpos=90,L=\hbox{},x=\r*2.5253cm,y=\r*5.0cm]{v1}
\Vertex[Lpos=90,L=\hbox{},x=\r*2.5253cm,y=\r*6.0cm]{v11}
\Vertex[L=\hbox{},x=\r*3.93933982822018cm,y=\r*3.0618cm]{v22}
\Vertex[L=\hbox{},x=\r*6.06066017177982cm,y=\r*3.0618cm]{v23}
\Vertex[L=\hbox{},x=\r*4.0006cm,y=\r*0.0cm]{v3}
\Vertex[L=\hbox{},x=\r*
3.25149038935436cm,y=\r*1cm]{v31}
\Vertex[Lpos=180,L=\hbox{},x=\r*0.942cm,y=\r*0cm]{v4}
\Vertex[Lpos=180,L=\hbox{},x=\r*1.69650961064564cm,y=\r*1cm]{v41}
\Vertex[Lpos=180,L=\hbox{},x=\r*0.0cm,y=\r*4.18416017177982cm]{v5} 
\Vertex[L=\hbox{},x=\r*-1.0606601717798cm,y=\r*3.12350000000000cm]{v51}
\Vertex[L=\hbox{},x=\r*1.06066017177982cm,y=\r*3.12350000000000cm]{v52}
%
\Edge[](v4)(v3)
\Edge[](v5)(v4)
\Edge[](v1)(v5)

\Edge[](v22)(v23)

\Edge[](v1)(v22)
\Edge[](v1)(v23)

\Edge[](v3)(v31)

\Edge[](v22)(v31)
\Edge[](v23)(v31)
\Edge[](v4)(v31)
\Edge[](v41)(v31)

\Edge[](v23)(v11)

\Edge[](v3)(v22)

\Edge[](v4)(v41)

\Edge[](v5)(v51)
\Edge[](v5)(v52)
\Edge[](v51)(v52)

\Edge[](v1)(v51)
\Edge[](v1)(v52)

\Edge[](v1)(v11)
\Edge[](v11)(v22)

\Edge[](v11)(v5)
\Edge[](v11)(v52)

\Edge[](v4)(v52)
\Edge[](v4)(v51)

\Edge[](v41)(v5)
\Edge[](v41)(v51)
\Edge[](v41)(v52)

\Edge[](v41)(v3)
\Edge[](v3)(v23)
\tikzset{EdgeStyle/.append style = {bend right}}

\Edge[](v11)(v51)
\tikzset{EdgeStyle/.append style = {bend left}}
\end{tikzpicture}}}
\caption{All $6$-vertex-critical (gem, co-gem)-free graphs.}\label{fig:6critgemcogem} 
\end{figure}

\section{Conclusion}\label{sec:conclusion}
In this paper we showed that for all $k\ge 1$ and $\ell\ge 0$ that there are only finitely many $k$-vertex-critical $(\forbid)$-free graphs and only finitely many $k$-vertex-critical (gem, co-gem)-free graphs. Our results imply the existence of new polynomial-time certifying algorithms to determine the $k$-colorability of graphs in each respective family. This has substantially reduced the number of open cases for a complete dichotomy on which graphs $H$ are there only finitely many $k$-vertex-critical $H$-free graphs for all $k$. The only remaining open case is now just $H=P_4+\ell P_1$ (Generalized Problem~2 in~\cite{CameronHoangSawada2020}). 

The structural characterization we were able to prove for $k$-vertex-critical (gem, co-gem)-free graphs and a similar previous for $(P_5, \overline{P_5})$-free graphs~\cite{Dhaliwal2017} lead us to pose the following open problem in analogy to the similar open problem on $\chi$-bounding due to Gy\'{a}rf\'{a}s~\cite{Gyarfas1987}.

\begin{problem}\label{prob:fco-Fcrit}
For which forests $F$ is there only a finite number of $k$-vertex-critical $(F,\overline{F})$-free graphs for all $k$?
\end{problem}
\noindent A natural place to start on Problem~\ref{prob:fco-Fcrit} would be to consider other forests $F$ of order $5$. In particular, $F=P_2+P_3$ would be of interest given the very recent work of Char and Kartick~\cite{CharKarthick2022}.
\bibliographystyle{abbrv}
\bibliography{P3+ell-2022-06-03}

\end{document}